\documentclass{birkjour}
 \newtheorem{thm}{Theorem}[section]
 \newtheorem{cor}[thm]{Corollary}
 \newtheorem{lem}[thm]{Lemma}
 
 \theoremstyle{definition}
 
 \theoremstyle{remark}

 \numberwithin{equation}{section}
 \usepackage{algorithmic}
\usepackage{matlab-prettifier}

\begin{document}

%
%
%
%
%
%
%
%
%

\title{A convex cover for closed unit curves has area at least 0.1}

\author{Bogdan Grechuk}

\address{
School of Mathematics and Actuarial Science\\
University of Leicester, Leicester\\
LE1 7RH\\
United Kingdom}

\email{bg83@leicester.ac.uk}

\author{Sittichoke Som-am}
\address
{School of Mathematics and Actuarial Science\\
University of Leicester, Leicester\\
LE1 7RH\\
United Kingdom}
\email{ssa41@leicester.ac.uk}


\date{\today}

\begin{abstract}
We improve a lower bound for the smallest area of convex covers for closed unit curves from $0.0975$ to $0.1$, which makes it substantially closer to the current best upper bound $0.11023$. We did this by considering the minimal area of convex hull of circle, line of length $\frac{1}{2}$, and rectangle with side $0.1727\times 0.3273$. By using geometric methods and the Box search algorithm, we proved that this area is at least $0.1$. We give informal numerical evidence that the obtained lower bound is close to the limit of current techniques, and substantially new idea is required to go significantly beyond $0.10044$.
\end{abstract}

\maketitle
\section{Introduction}

Moser's worm problem is one of many unsolved problems in geometry posed by Leo Moser [11] in 1966 (see [12] for a new version). It asked about the smallest area of region $S$ in the plane which can be rotated and translated to cover every unit arc. While this problem is still open, lower and upper bounds for this smallest area are available in the literature. The current best upper bound $0.260437$ was established by Norwood and Poole [13]. For the lower bound, it is known only that the area of $S$ is strictly positive [10].

This problem can be modified by restricting region $S$ (e.g. triangle, rectangle, convex, non-convex, ect.) or restricting unit curves (e.g. closed, convex, ect.). Numerous problems in the worm family have been solved prior, see Wetzel [17] for a survey.

If we require $S$ to be convex, and cover arbitrary unit curves, the existence of the solution is shown by Laidecker and Poole [9] using Blachke Selection Theorem. In 2019, Panraksa and Wichiramala [14] showed that the Wetzel's sector is a cover for unit arcs which has an area of $\pi/12 \approx 0.2618$. This cover is the current upper bound for this problem, whereas the current lower bound $0.232239$ was found by Khandhawit, Sriswasdi and Pagonakis [8] in 2013. 

The problem of finding smallest convex region $S$ to cover all (not necessarily unit) curves of \emph{diameter} one was posed by Henri Lebesgue in 1914 and is known as Lebesgue's universal covering problem. In this case, the current best lower and upper bounds are $0.832$ and $0.8440936$ due to Peter Brass and Mehrbod Sharifi [1] and Philip Gibbs [6], respectively.

In this paper, we consider the problem of covering {\bf \emph{closed unit}} curves by a {\bf \emph{convex}} region of the smallest area.
 
In 1957, H.G. Eggleton [3] showed that the equilateral triangle of side ${\displaystyle \frac{\sqrt{3}}{\pi}}$ is the smallest triangle which can cover all unit closed curves and its area is $\frac{3\sqrt{3}}{4\pi^2}\approx 0.13162$. Twenty five years later, the smallest rectangle with lengths $\frac{1}{\pi}$ and $\sqrt{\frac{1}{4}-\frac{1}{\pi^2}}$ is found by Shaer and Wetzel [15] and its area is about $0.12274$.       

Furedi and Wetzel [5] decreased this area by clipping one corner of this rectangle with area about $0.11222$. Also, they modified this pentagon to a curvilinear hexagon which has an area of less than $0.11213$. In 2018, Wichiramala [18] showed that the opposite corner of the pentagon can be clipped to be a hexagonal cover with area less than $0.11023$ which is the best currently known bound. 

Our work focuses on a lower bound for this problem. Chakerian and Klamkin [2] applied F{\'a}ry and R{\'e}dei's theorem [4] to find the first lower bound which is $0.0963275$ by considering the convex hull of circle and line segment in 1973. Further, Furedi and Wetzel [5] 
showed that the minimum area of convex hull of circle and the $0.0261682\times0.4738318$ rectangle has area $0.0966675$. Moreover, they replaced this rectangle by curvilinear rectangle and give a new lower bound about $0.096694$.

To improve this bound, one may consider the minimal convex hull of three given closed curves. However, in this case F{\'a}ry and R{\'e}dei's theorem [4] can not be applied to find smallest area analytically, which complicates the analysis and call for the mix of geometric and numerical methods. Som-am [16] applied Brass grid method [1] to prove that the area of convex hull of circle, line segment ,and equilateral triangle is at least $0.096905$. In 2019, Grechuk and Som-am [7] used the Box search method to show that the minimal-area convex hull of the rectangle of perimeter $1$, circle with perimeter $1$, and the equilateral triangle with perimeter $1$, is greater than $0.0975$. 
Thus, if we denote $\alpha$ to be the area of smallest cover of closed unit curves, then
$$
0.0975 \leq \alpha \leq 0.11023.
$$ 

In this paper, we use geometric methods combined with the Box search algorithm to prove that the area of convex cover for the line segment, circle, and certain rectangle of perimeter $1$ is at least $0.1$.

\begin{thm}\label{th:main}
The area of convex cover $S$ for circle of perimeter $1$, line of length $1/2$, and rectangle of size $0.1727\times0.3273$ is at least $0.1$.
\end{thm}

By Theorem \ref{th:main}, we have 
$$
0.1 \leq \alpha,
$$
which improves the previous bound $0.0975$. In fact, the true minimal area of $S$ in Theorem \ref{th:main} is about $0.10044$, but we have rounded the bound to $0.1$ to simplify the proof. 

The choice of  line segment, circle, and rectangle in Theorem 1 was not an arbitrary or lucky one. In fact, we have used a systematic search for three shapes which gives the best possible lower bound. 

The organization of this paper is as follows. A systematic search for three shapes which maximal possible area of minimal convex hull is presented in Section \ref{sec:search}. The remaining sections are devoted to the proof of Theorem \ref{th:main}. Section \ref{sec:geom} proves some geometric lemmas. Section \ref{sec:compu} describes numerical box-search algorithm and computational results. The proof of Theorem \ref{th:main} is finished in Section \ref{sec:main}. We give some conclusions in Section \ref{sec:con}.

\section{A systematic search for shape}\label{sec:search}

Theorem \ref{th:main} establishes a new lower bound for the area of a convex cover for closed unit curves, by proving that a convex hull of three such curves, circle, line, and certain rectangle, must have area at least $0.1$. In fact, the choice of these three curves was not an arbitrary or lucky one, but was the result of a systematic numerical search, which we will describe in this section.

This section contains neither proofs nor any form of rigorous analysis, and the reader who is interested only in the proof of Theorem \ref{th:main} can go straight to the next sections. However, this section is important to understand how we selected three curves considered in Theorem \ref{th:main}, and provides some guide were to not look when trying to substantially improve our lower bound.  

For each specific set of unit curves, we write a program which finds minimum area of a convex hull of these curves, where the minimum is taken with respect to all translations and rotations of the curves. We then modify the curves to make this minimal as large as possible. This result is a maximin problem with highly non-smooth and non-convex objective function. We used Matlab function, patternsearch, with random initial data to try to find the global optimal value in various cases. 

Let $\alpha$ be the area of smallest cover of closed unit curves. We will approximate curves by polygons, and it is clear that it suffices to consider convex polygons only. Let us start with 2 polygons. Let $F(X_1,X_2)$ be the area of convex hull of polygons $X_1$ and $X_2$. Let $\mathcal{T}$ be the set of all orientation-preserving motions $T$ of the plane (that is, all possible compositions of rotations and translations). Then, given any two polygons $X_1$ and $X_2$ with unit perimeter, the solution of the minimization problem
\begin{equation}\label{eq:min2objects} 
\min_{T \in \mathcal{T}} F(X_1,T(X_2))
\end{equation}
is the lower bound for $\alpha$. 

Now, our aim to find polygons $X_1$ and $X_2$ for which this lower bound is as large as possible. Let $N_1$ and $N_2$ be a number of vertices in polygons $X_1$ and $X_2$ respectively. We assume that $N_1$ and $N_2$ are fixed but $X_1$ and $X_2$ can vary. Let $\mathcal{X}(N)$ be the set of all convex polygons with $N$ vertices and unit perimeter. We consider optimization problem
\begin{equation}\label{eq:2objects} 
b(N_1, N_2) = \max_{X_1 \in \mathcal{X}(N_1),X_2 \in \mathcal{X}(N_2)}\min_{T \in \mathcal{T}} F(X_1,T(X_2)).
\end{equation}
It is clear that for any $N_1$, $N_2$,
$$
b(N_1, N_2) \leq \alpha,
$$
or, in words, $b(N_1, N_2)$ is a lower bound for $\alpha$.

\subsection{Numerical results for 2 objects}

First, we construct a Matlab function NN2per to solve the minimization problem \eqref{eq:min2objects}. The input of the function is $2N_1 + 2N_2$ coordinates of vertices of $X_1$ and $X_2$. The function first calculates the perimeters of the polygons, and scale them to make the perimeters to be equal to $1$. Then it applies the  motion $T$ to $X_2$, which is described by three parameters: vector of translation $(x_1,y_1)$ and angle of rotation $\theta_1$, and use Matlab function convhull to estimate $F(X_1,T(X_2))$. We then use function MultiStart in Matlab to solve the minimization problem of the resulting convex hull area over parameters $x_1, y_1, \theta_1$.   

Second, we solve the maximization problem \eqref{eq:2objects} by finding maximal possible output of function NN2per for fixed $N_1$ and $N_2$. We start with some initial configuration (for example, $X_1$ is regular $N_1$-gon and $X_2$ regular $N_2$-gon), and then use the patternsearch function. 

We repeat this procedure for various small values of $N_1$ and $N_2$. Specifically, we consider the cases of a line and a triangle (2+3 points), two triangles (3+3 points), a triangle and a rectangle (3+4 points), two rectangles (4+4 points), and so on. The results are shown in Table \ref{table:1}.

\begin{table}
\begin{center}
\begin{tabular}{ |c|c|c| } 
 \hline
 Type ($n+m$ points) & Optimal area & Time (sec) \\ 
 \hline
 2+3 & 0.072169 & 5889 \\ 
 3+3 & 0.072375  & 7995 \\ 
 3+4 & 0.085377 & 14080 \\ 
 4+4 & 0.085377 & 17370 \\ 
3+5   & 0.087902 & 17296 \\ 
5+5   & 0.087902 & 23684 \\ 
9+19 & 0.095790 & 239828 \\ 
 20+20 & 0.096120 & 240258 \\ 
9+50   & 0.096595   & 241172 \\ 
11+50 & 0.096605 & 251041 \\ 
 \hline
\end{tabular}
\end{center}
\caption{The numerical series for 2 objects}
\label{table:1}
\end{table}
From the Table \ref{table:1}, it can be seen that the maximum lower bound we found is $0.0966053$, see Figure \ref{fn1}. This is close to the best known lower bound $0.0966675$ found from considering 2 objects [5]. We can conclude that in this case we did not improve the existing best lower bound but almost recovered the best result in the literature. 
 
\begin{figure}
\begin{centering}
\includegraphics[scale=0.6]{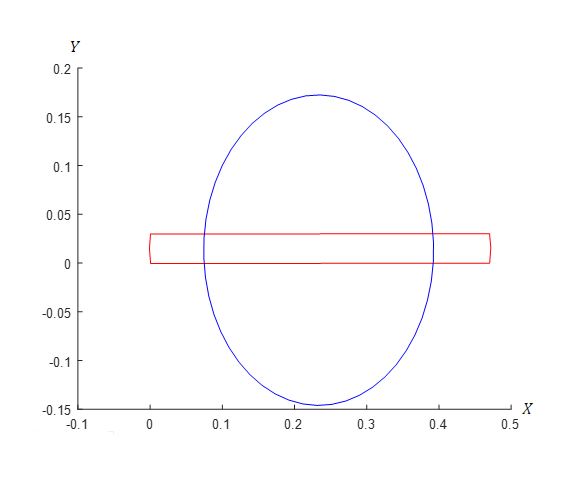}
\par\end{centering}
\caption{The configuration for the $11$ points and $50$ points which has area of $0.0966053$}
\label{fn1}
\end{figure}

\subsection{Numerical results for 3 objects} 

For three objects, we use the same idea as for 2 objects. Let $X_3$ be the third polygon with $N_3$ vertices which can be translated and rotated. We then solve the maximin optimization problem
\begin{equation}\label{eq:3objects} 
b(N_1, N_2, N_3) = \max_{X_1 \in \mathcal{X}(N_1),X_2 \in \mathcal{X}(N_2),X_3 \in \mathcal{X}(N_3),}\min_{T_1 \in \mathcal{T}, T_2 \in \mathcal{T}} F(X_1,T_1(X_2),T_2(X_3)),
\end{equation}
where $F$ now denotes the area of convex hull of three polygons. Again, $b(N_1, N_2, N_3)$ is the lower bound for $\alpha$.

For the internal minimization problem, we now optimize over six parameters which are $x_1,y_1,\theta_1$ and $x_2,y_2,\theta_2$ for translation and rotation of first and second polygons respectively. The maximization problem is now with respect to $2N_1 + 2N_2 + 2N_3$ coordinates of polygon's vertices.

We start the cases of line and two triangles (2+3+3), three triangles (3+3+3), two triangles and a rectangle (3+3+4) and so on. We can see the results in Table \ref{table:2}.

\begin{table}
\begin{center}
\begin{tabular}{ |c|c|c| } 
 \hline
 Type ($n+m+k$ points) & Optimal area & Time (sec) \\ 
 \hline
 2+3+3 & 0.072169 & 25741 \\ 
 3+3+3 & 0.073086  & 29178 \\ 
 2+3+4 & 0.087867 & 35518 \\ 
 3+3+4 & 0.087887 & 38549 \\ 
 3+3+5   & 0.088478 & 51306 \\ 
 10+10+10 & 0.093546 & 208439 \\
2+4+50   & 0.1004   & 327082 \\ 
4+4+50 & 0.100403 & 505971 \\ 
 7+7+50 & 0.100417    & 1022268 \\ 
 \hline
\end{tabular}
\end{center}
\caption{The numerical series for 3 objects}
\label{table:2}
\end{table}

\begin{figure}
\begin{centering}
\includegraphics[scale=0.6]{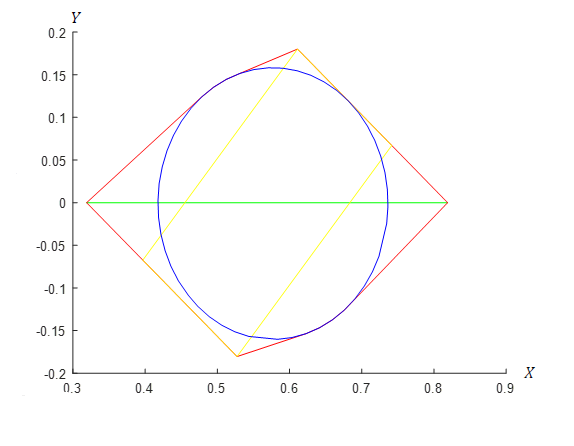}
\par\end{centering}
\caption{The configuration for the $7$ points, $7$ points and $50$ points which has area of $0.100417$}
\label{fn2}
\end{figure}

Table \ref{table:2} shows that the maximum lower bound for $\alpha$ we found is $0.100417$ which is better than the current bound $0.0975$ [7]. From the Figure \ref{fn2}, the three polygons corresponding to this lower bound approximates a circle of perimeter $1$, a line of length $0.5$ and a rectangle of size $0.1727\times0.3273$. 

Inspired by this result, we also try to find three objects in the form: circle, line, and $n$-gon. We represent circle as regular $500$-gon, line as a $2$-gon, and then increased the number of vertices of the third polygon from $3$ to $10$. The results are presented in Table \ref{table:3}.

Because $n-1$-gon is a special case of $n$-gon with coinciding vertices, the lower bound improves by definition, and it is best for $n=10$. However, Figure \ref{fn3} shows that the resulting $10$-gon is very similar to the same rectangle of size $0.1727\times0.3273$ which we have obtained in the $N_1=7, N_2=7, N_3=50$ experiment in Table 2. The resulting bound $0.100473$ is also very close to the bound $0.100417$ in Table 2.

\begin{table}
\begin{center}
\begin{tabular}{ |c|c|c| } 
 \hline
 Type ($n$ points) & Optimal area & Time (sec) \\ 
 \hline
 3 & 0.0970429   & 13508  \\ 
 4 & 0.1003  & 28865 \\ 
 5 & 0.100304 & 33652 \\ 
 6 & 0.100374 & 54979\\ 
 7   & 0.100386 & 80077 \\ 
8   & 0.100390   & 81285 \\ 
9 & 0.100418 & 119103 \\ 
 10 & 0.100473    & 123350 \\ 
 \hline
\end{tabular}
\end{center}
\caption{The numerical series for 3 objects when 2 objects are fixed}
\label{table:3}
\end{table}

\begin{figure}
\begin{centering}
\includegraphics[scale=0.7]{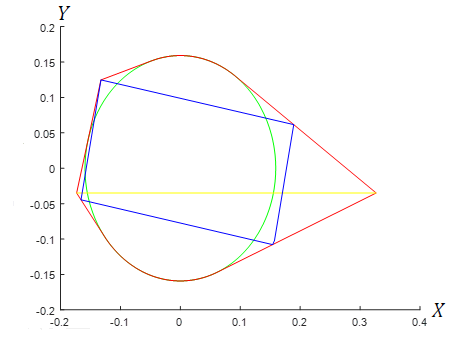}
\par\end{centering}
\caption{The configuration for the $500$-gon, a strength line and $10$ points which has area of $0.100473$}
\label{fn3}
\end{figure}


Based on the numerical experiments in this section, we
\begin{itemize}
\item Have found a line-rectangle-circle configuration, whose convex hull is, numerically, always grater than $0.1$. This significantly improves the previous best lower bound $0.0975$. 
\item Conclude that no other ``simple'' configuration of 3 objects gives a substantially better bounds.
\end{itemize}

In the following section, we give a rigorous prove for the first conclusion, and this is the main result of this work. The second conclusion is informal and has no proof.

\section{Geometric analysis}\label{sec:geom}

Let $C$ be a circle of perimeter $1$, $R$ be a rectangle of size $u \times v$ where $u=0.1727$ and $v=\dfrac{1}{2}-u$, and $L$ be line of length $\dfrac{1}{2}$. 

Let us fix the center of circle to be $C_0(0,0)$. Let $F$ be a regular $500$-gon inscribed in $C$, such that the sides of $R$ are parallel to some longest diagonals of $F$. Let $X$ be a union $F \cup R \cup L$. $X$ is called a \emph{configuration}. Let $\mathcal{H}(X)$ be the convex hull of $X$, and $\mathcal{A}(X)$ be the area of $\mathcal{H}(X)$. Our aim is to find a configuration $X$ with the smallest $\mathcal{A}(X)$.

Let $R_0(x_1,y_1)$ be the center of $R$. By the symmetry of circle, we may assume that $x_1,y_1\geq 0$. We have the vertices of $R$ are $R_{1}\left(x_{1}-v/2,
y_{1}+u/2\right)$, 
$R_{2}\left(x_{1}+v/2,y_{1}+u/2\right)$, 
$R_{3}\left(x_{1}+v/2,y_{1}-u/2\right)$, 
and $R_{4}\left(x_{1}-v/2,y_{1}-u/2\right)$. Let $L_0 (x_{2},y_{2})$ be the center of $L$ and $\theta$ be the angle between X axis and $L_0L_2$, see Figure \ref{f1}.

\begin{figure}
\begin{centering}
\includegraphics[scale=0.7]{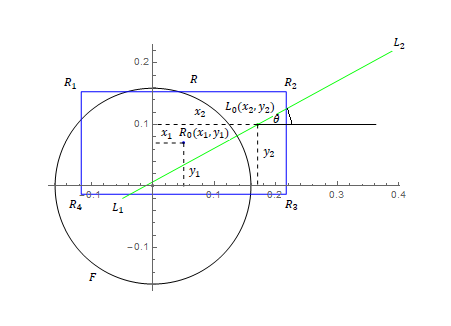}
\par\end{centering}
\caption{The configuration $X$}
\label{f1}
\end{figure}

The vertices of $L$ are $L_{1}(x_{2}+\frac{1}{4}\cos(\theta+\pi),y_{2}+\frac{1}{4}\sin(\theta+\pi))$ and 
$L_{2}(x_{2}+\frac{1}{4}\cos(\theta),y_{2}+\frac{1}{4}\sin(\theta))$. Note that any configuration $X$ is represented by $x_{1},y_{1},x_{2},y_{2}$, and $\theta$. Clearly, $0\leq\theta\leq\pi$

We define the function $f:\mathbb{R}^{5}\rightarrow\mathbb{R}$ by mapping vector $(x_{1},y_{1},x_{2},y_{2},\theta)$ to  $\mathcal{A}(X)$. Clearly, $f$ is a continuous function. Since $F$ is a subset of $C$, it suffices to show that 
$$
f(x_{1},y_{1},x_{2},y_{2},\theta) \geq 0.1, \quad \forall \, x_{1}, y_{1}, x_{2}, y_{2}, \theta.
$$
This would immediately imply Theorem \ref{th:main}.

Next, we will apply result of Fary and Redei [4], and Lemma 1 and Corollary 2 in [7] to bound parameters $x_{1},y_{1},x_{2},y_{2}$ and $\theta$.

\begin{lem}\label{lem:conditions}
Let $Z$ be a region of points $z=(x_{1},y_{1},x_{2},y_{2},\theta)$ in ${\mathbb R}^5$ satisfying the inequalities
$
0\leq x_{1}\leq0.0741,\,\, 0\leq y_{1}\leq0.0976,\,\, -0.148\leq x_{2}\leq0.148,\,\, -0.148\leq y_{2}\leq 0.148,\,\, 0\leq\theta\leq{\displaystyle \pi}.$
If $f(z) > 0.1$ for all $z \in Z$, then in fact $f(z)> 0.1$ for all $z \in {\mathbb R}^5$.
\end{lem}
\begin{proof}
Let $\psi(x_{1}, y_{1})$ be the area of convex hull of $F$ and $R$ only. By  [4], $\psi$ is a convex function in both coordinates. By symmetry, $\psi(x_{1}, y_{1}) = \psi(x_{1}, -y_{1})$. This combined with convexity implies that
$$
\psi(x_{1}, y_{1}) \geq \psi(x_{1}, 0) \geq \psi(0.0741, 0) > 0.1,
$$  
whenever $x_{1} \geq 0.0741$. Similarly,
$$
\psi(x_{1}, y_{1}) \geq \psi(0, y_{1}) \geq \psi(0, 0.0976) > 0.1,
$$ 
whenever $y_{1} \geq 0.0976$.

Let $\Phi(x_2,y_2)$ be the area of convex hull of $F$ and $L$ only. By  [4], if $L$ moves in direction $\overrightarrow{L_1L_2}$, $\Phi(x_2,y_2)$ attains minimum when $C_0L_0$ is perpendicular to $L_1L_2$.  
Assume that $|x_{2}|\geq 0.148$ or $|y_{2}|\geq 0.148$. Then $\sqrt{x_{2}^{2}+y_{2}^{2}} \geq 0.148$. 

\begin{figure}[ht]
\begin{centering}
\includegraphics[scale=0.6]{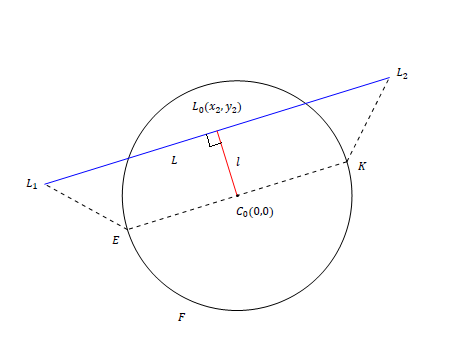}
\par\end{centering}

\caption{$EKL_{1}L_{2}$ and $l$}
\label{fig:f2}
\end{figure}

 Let $l$ be the line segment between $C_{0}$ and $L_{0}$. Next, let points  $E,K \in F$ be such that line segments $C_{0}L_{0}$ and $EK$ are perpendicular, see Figure \ref{fig:f2}. Then $EKL_{1}L_{2}$ is trapezoid with bases lengths $|EK|\geq 2r \cos\left(\frac{\pi}{500}\right)$ and $|L_{1}L_{2}|=0.5$, where $r=\frac{1}{2\pi}$. The area of $F$ is $S(F)=500\frac{r^2}{2}\sin\left(\frac{2\pi}{500}\right)$.
Thus, 
$$
\mathcal{A}(X)>{\displaystyle\frac{1}{2}\left(\frac{1}{2}+2r \cos\left(\frac{\pi}{500}\right)\right)\sqrt{x_{2}^{2}+y_{2}^{2}}}+\frac{S(F)}{2} >0.1
$$
\end{proof}

\begin{lem}\label{lem:bigrange}
Either $f(z)> 0.1$, or $F \cup L \cup R$ is a subset of a rectangle with side lengths $0.439 \times 0.636$. 
\end{lem}
\begin{proof}
By Lemma \ref{lem:conditions}, we can assume that $z = (x_{1},y_{1},x_{2},y_{2},\theta) \in Z$. 
Let $Y_1$ and $Y_2$ be the points of configuration $X=F \cup L \cup R$ with the lowest and highest $y$-coordinates $y^*_1$ and $y^*_2$, respectively. Because $0 \leq y_1 \leq 0.0976, Y_1$ is below $R$. Let $h_1, h_2$ be the height from $Y_1, Y_2$ to $R$, respectively. $h_2=0$ if $Y_2$ is below or on $R_1R_2$. Let $y^*_2-y^*_1>0.439$. We have $\mathcal{A}(X) > u(\frac{1}{2}-u)+\frac{1}{2}(\frac{1}{2}-u)(h_1+h_2)=u(\frac{1}{2}-u)+\frac{1}{2}(\frac{1}{2}-u)(y^*_2-y^*_1-u)>0.1$ see Figure \ref{fig:f3}

\begin{figure}[ht]
\begin{centering}
\includegraphics[scale=0.3]{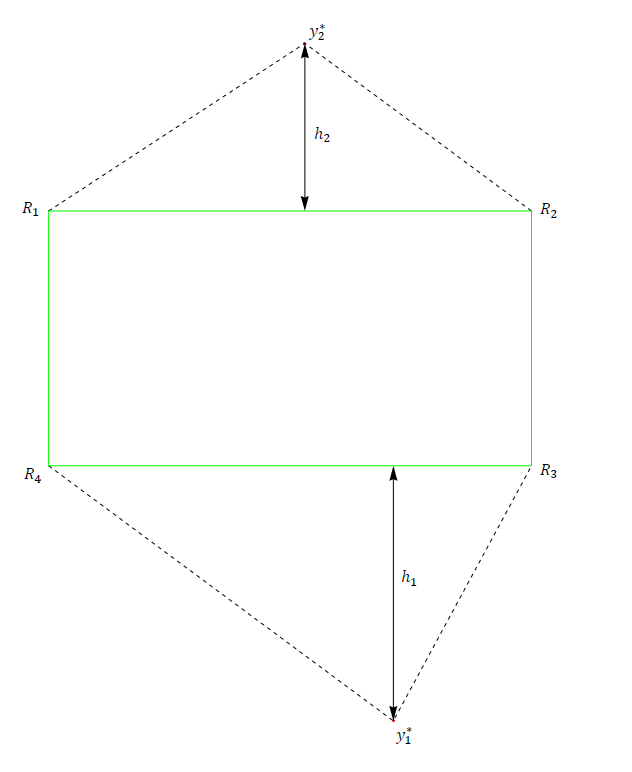}
\par\end{centering}

\caption{The configuration of $y^*_1,y^*_2$ and $R$}
\label{fig:f3}
\end{figure}

Let $X_1$ and $X_2$ be the points of configuration $X=F \cup T \cup R$ with the lowest and highest $x$-coordinates $x^*_1$ and $x^*_2$, respectively. Let $w=(x,y)$ be any point in $X$. If $w$ belongs to $F$, then $|x|<0.159$. Lemma \ref{lem:conditions} implies that if $w$ is any point on rectangle, then $|x|<0.2378$ and if $w$ is any point on line, then $|x|<0.398$. If $X_1,X_2 \in L$, then $x^*_2-x^*_1\leq 0.5$. Otherwise, $x^*_2-x^*_1\leq |x^*_2|+|x^*_1|<0.398+0.2378<0.636$ see Figure \ref{fig:f4}. 

\begin{figure}[ht]
\begin{centering}
\includegraphics[scale=0.7]{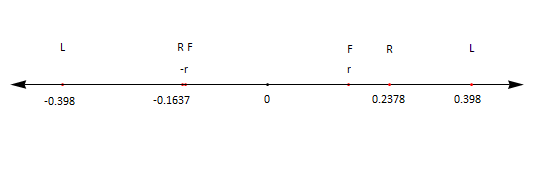}
\par\end{centering}
\caption{The line shows the optimum possible position of $F, R, L$ }
\label{fig:f4}
\end{figure}

\end{proof}

We prove Lipschitz continuity of $f$ in $Z$ in the following lemma

\begin{lem}\label{lem:cont}
For every $(x_{1},y_{1},x_{2},y_{2},\theta) \in Z$, and any $\epsilon_i \geq 0$, $i=1,\dots,5$,
$$
|f(x_{1}+\epsilon_{1},y_{1}+\epsilon_{2},x_{2}+\epsilon_{3},y_{2}+\epsilon_{4},\theta+\epsilon_{5})-f(x_{1},y_{1},x_{2},y_{2},\theta)| \leq \sum_{i=1}^5 \epsilon_i C_i,
$$
with constants 
$C_{1}=0.306$, $C_{2}=0.443$ , $C_{3}=0.392$ , $C_{4}=0.449$, and $C_{5}=0.115$.

\end{lem} 
\begin{proof}
Let $g:{\mathbb R}\to{\mathbb R}$ be a convex function on ${\mathbb R}$. 
If
$$
C = \max\left[\lim\limits_{t\to -\infty} \frac{g(t)}{t}, \lim\limits_{t\to +\infty} \frac{g(t)}{t}\right] < \infty,
$$
then
$$
|g(t+\epsilon)-g(t)| \leq C \epsilon, \quad \forall t, \, \forall \epsilon > 0. 
$$
see Lemma 5 in [7].

Let $g(x_1) = f(x_{1},y_{1},x_{2},y_{2},\theta)$, where $y_{1},x_{2},y_{2},\theta$ are fixed.
We have
$$
\lim\limits_{x_1\to -\infty} \frac{g(x_1)}{x_1} = \lim\limits_{x_1\to +\infty} \frac{g(x_1)}{x_1} \leq \frac{0.439+u}{2} < C_1,
$$
where $0.439$ comes from Lemma \ref{lem:bigrange}, while $u$ is the height of $R$,
see Figure  \ref{fig:f5}.
\begin{figure}[ht]
\begin{centering}
\includegraphics[scale=0.7]{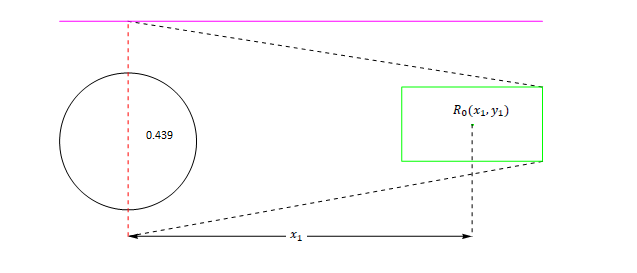}
\par\end{centering}

\caption{The ratio between $g(x_1)$ and $x_1$ when $x_1\to+\infty$}
\label{fig:f5}
\end{figure}

Hence,
$$
|f(x_{1}+\epsilon_{1},y_{1},x_{2},y_{2},\theta)-f(x_{1},y_{1},x_{2},y_{2},\theta)|\leq C_{1}\epsilon_{1}.
$$

Similarly, with $g(y_1) = f(x_{1},y_{1},x_{2},y_{2},\theta)$ for fixed $x_{1},x_{2},y_{2},\theta$, 
$$
\lim\limits_{y_1\to -\infty} \frac{g(y_1)}{y_1} = \lim\limits_{y_1\to +\infty} \frac{g(y_1)}{y_1} \leq \frac{(0.5-u)+(x_2+0.25+r)}{2} < C_2,
$$
where $r=1/2\pi$ and $|x_2| \leq 0.148$,
see Figure \ref{fig:f6},
while with $g(x_2) = f(x_{1},y_{1},x_{2},y_{2},\theta)$,
$$
\lim\limits_{x_2\to -\infty} \frac{g(x_2)}{x_2} = \lim\limits_{x_2\to +\infty} \frac{g(x_2)}{x_2} \leq \frac{0.439 + (y_1+u/2+r)}{2} < C_3.
$$
where $0.439$ come from Lemma \ref{lem:bigrange} and $0 \leq y_1 \leq 0.0976$
see Figure \ref{fig:f7}. Next, with $g(y_2) = f(x_{1},y_{1},x_{2},y_{2},\theta)$,
$$
\lim\limits_{y_2\to -\infty} \frac{g(y_2)}{y_2} = \lim\limits_{y_2\to +\infty} \frac{g(y_2)}{y_2} \leq \frac{0.5 + (x_1+(0.5-u)/2+r)}{2} < C_4.
$$
where $0 \leq x_1 \leq 0.0741$,
\begin{figure}[ht]
\begin{centering}
\includegraphics[scale=0.5]{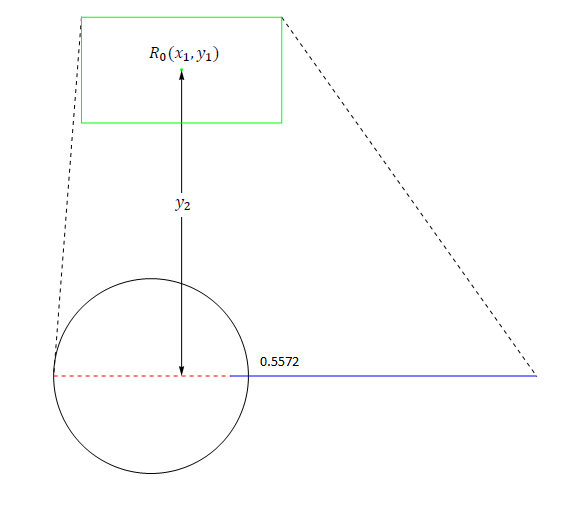}
\par\end{centering}

\caption{The ratio between $g(y_1)$ and $y_1$ when $y_1\to+\infty$}
\label{fig:f6}
\end{figure}

\begin{figure}[ht]
\begin{centering}
\includegraphics[scale=0.5]{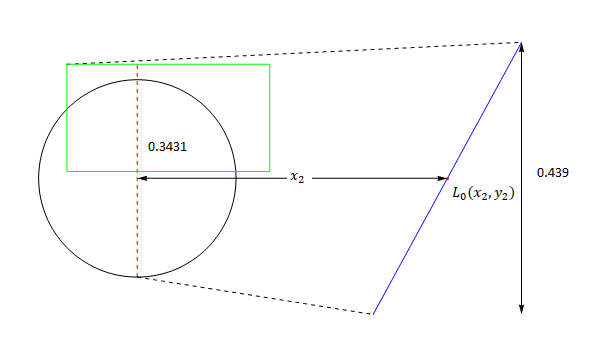}
\par\end{centering}

\caption{The ratio between $g(x_2)$ and $x_2$ when $x_2\to+\infty$}
\label{fig:f7}
\end{figure}

\begin{figure}[ht]
\begin{centering}
\includegraphics[scale=0.3]{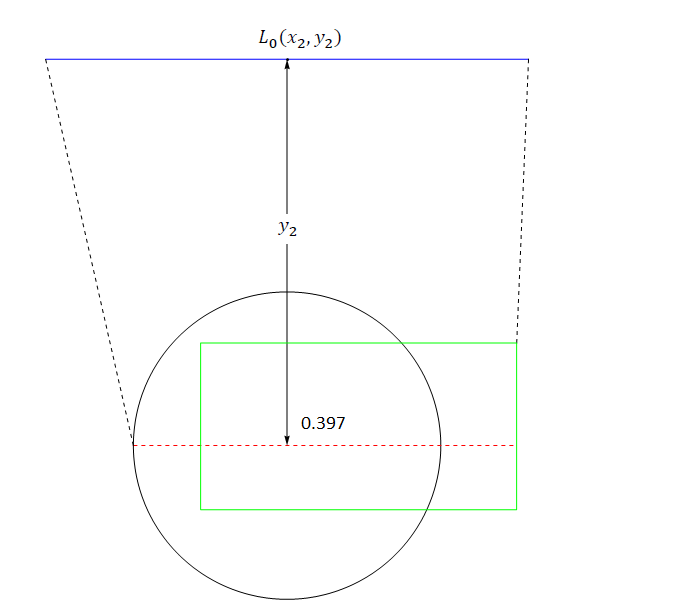}
\par\end{centering}

\caption{The ratio between $g(y_2)$ and $y_2$ when $y_2\to+\infty$}
\label{fig:f8}
\end{figure}

see Figure \ref{fig:f7}. This implies that
$$
|f(x_{1},y_{1}+\epsilon_{2},x_{2},y_{2},\theta)-f(x_{1},y_{1},x_{2},y_{2},\theta)|\leq C_{2}\epsilon_{2},
$$

$$
|f(x_{1},y_{1},x_{2}+\epsilon_{3},y_{2},\theta)-f(x_{1},y_{1},x_{2},y_{2},\theta)|\leq C_{3}\epsilon_{3},
$$
and
$$
|f(x_{1},y_{1},x_{2},y_{2}+\epsilon_{4},\theta)-f(x_{1},y_{1},x_{2},y_{2},\theta)|\leq C_{4}\epsilon_{4}.
$$

Finally, we need to prove that 
\begin{equation}\label{eq:C5}
|f(x_{1},y_{1},x_{2},y_{2},\theta+\epsilon_{5})-f(x_{1},y_{1},x_{2},y_{2},\theta)|\leq C_{5}\epsilon_{5}.
\end{equation} 
To prove the bound for $C_5$, we need the following claim.

{\bf Claim 1.} The diameter $d(\mathcal{F\cup R})$ of $\mathcal{F\cup R}$ is less than $0.45976$. 

By Lemma \ref{lem:conditions}, the diameter $d(\mathcal{F\cup R})$ is maximal possible if $R_0=(0.0741,0.0976)$. Then $R_2=(0.23775,0.18395)$. Let $F_1\in F$ be a point which $d(x,R_2)$ is maximum for all $x\in F$ see Figure \ref{fig:f8}. By direct calculation, $|R_2 R_4|< 0.37007 < |R_2 F_1|$. Hence, the diameter of $\mathcal{F\cup R}$ is $|R_2 F_1| < 0.45976$.
\begin{figure}[ht]
\begin{centering}
\includegraphics[scale=0.4]{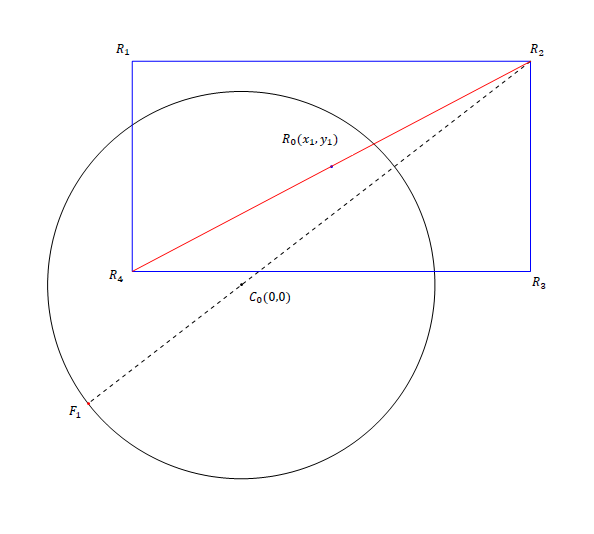}
\par\end{centering}

\caption{The longest between $R_2$ and $F$}
\label{fig:f8}
\end{figure}

Next, we will prove \eqref{eq:C5}. We note that if \eqref{eq:C5} holds for $\epsilon_5=\epsilon$, then it also holds for $\epsilon_5=2\epsilon$. Therefore it is sufficient to prove \eqref{eq:C5} only for sufficiently small $\epsilon_5$.  

Let $L'$ with endpoints $L'_1, L'_2$ be the line $L$ rotated around $L_0$ by angle $\epsilon_{5}$. Then $|L_1L'_1| = 2|L_0L_1|\sin(\epsilon_5/2) < 2|L_0L_1|(\epsilon_5/2)=|L_0L_1|\epsilon_5 = \frac{1}{4}\epsilon_5$. Similarly, $|L_2L'_2|< \frac{1}{4}\epsilon_5$.

By selecting $\epsilon_5$ sufficiently small, we can ensure that all vertices of polygons $\mathcal{H}(R,F,L)$ and $\mathcal{H}(R,F,L')$ coincides, except possible the endpoints of $L$ and $L'$.
Then area difference $|\mathcal{A}(R,F,L')-\mathcal{A}(R,F,L)|$ is bounded by the total area of four triangles $X_1L_1X_2$, $X_1L'_1X_2$, $X_3L_2X_4$, $X_3L'_2X_4$, where $X_i, i=1,2,3,4$ are vertices by polygon $\mathcal{H}(R,F,L)$ adjacent to $L_1,L_2$, see Figures \ref{fig:f9} and \ref{fig:f10}. 
 
Let $h_1,h_2$ be the height of triangle with respect to base $X_1X_2$. Let $h_3,h_4$ be the height of triangle with respect to base $X_3X_4$ see Figure \ref{fig:f10}. By claim 1, We have $|\mathcal{A}(R,F,L')-\mathcal{A}(R,F,L)|\leq |\frac{1}{2}h_1X_1X_2-\frac{1}{2}h_2X_1X_2|+ |\frac{1}{2}h_3X_3X_4-\frac{1}{2}h_4X_3X_4|=\frac{1}{2}X_1X_2|h_1-h_2|+\frac{1}{2}X_3X_4|h_3-h_4| < \frac{1}{2}X_1X_2|L_1L'_1|+\frac{1}{2}X_3X_4|L_2L'_2|\leq 2\times\frac{1}{2}d(\mathcal{F\cup R})\times\frac{1}{4}\epsilon_{5} \leq \frac{1}{4}\times0.45976< 0.115\epsilon_{5} = C_5 \epsilon_{5}$.

\begin{figure}
\begin{centering}
\includegraphics[scale=0.4]{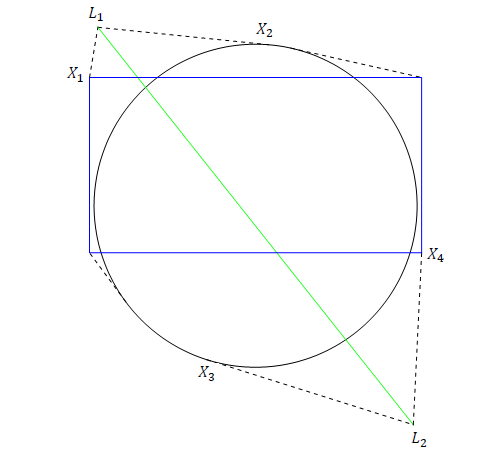}
\par\end{centering}

\caption{Polygon $\mathcal{H}(R, F, L)$ adjacent $L_1,L_2$}
\label{fig:f9}
\end{figure}

\begin{figure}
\begin{centering}
\includegraphics[scale=0.4]{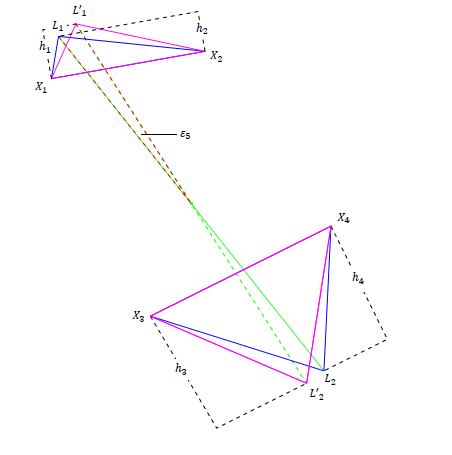}
\par\end{centering}

\caption{Four triangles which $L$ rotated by angle $\epsilon_5$}
\label{fig:f10}
\end{figure}

\end{proof}

\section{Computational results}\label{sec:compu}

Let $Z$ be a region in Lemma \ref{lem:conditions}. In this section, we prove that 
$$
f(z)=f(x_1,y_1,x_2,y_2,\theta)>0.1, \quad \forall z \in Z
$$ 
by using box-search method as described in \cite{grechuksomam2019convex}.
The method works as follows. Let $z^*$ be the center of a box $B$ which has the form $B=[a_{1},b_{1}]\times[a_{2},b_{2}]\times[a_{3},b_{3}]\times[a_{4},b_{4}]\times[a_{5},b_{5}]$. On every step, we check if 
\begin{equation}\label{eq:maincond}
f(z^*)-d_{1}C_{1}-d_{2}C_{2}-d_{3}C_{3}-d_{4}C_{4}-d_{5}C_{5}\geq 0.1,
\end{equation} 
where $d_i$ is a half of the length of $[a_i,b_i]$. If \eqref{eq:maincond} holds, then inequality $f(z) > 0.1$ holds for every $z \in B$ by Lemma \ref{lem:cont}. 

If \eqref{eq:maincond} does not hold, we will select the largest length, say $[a_1,b_1]$ and split $B$ into two boxes $B_1=[a_{1},(a_1+b_1)/2]\times[a_{2},b_{2}]\times[a_{3},b_{3}]\times[a_{4},b_{4}]\times[a_{5},b_{5}]$ and $B_2=[(a_1+b_1)/2,b_{1}]\times[a_{2},b_{2}]\times[a_{3},b_{3}]\times[a_{4},b_{4}]\times[a_{5},b_{5}]$. Then, if \eqref{eq:maincond} holds for both boxes $B_1$ and $B_2$, then $f(z) > 0.1$ holds for every $z \in B_1$ and for every $z \in B_2$, hence it holds for every $z \in B$. If \eqref{eq:maincond} does not hold for either $B_1$ or $B_2$ (or both), we subdivide the corresponding boxes again and proceed iteratively.

We start with $B=Z$, and, when the program halts, we are guaranteed that $f(z)>0.1, \forall z\in Z$.

Let us illustrate the first step of this procedure. 
By Lemma \ref{lem:conditions}, we start with $B=Z=[0,0.0741]\times[0,0.0976]\times[-0.148,0.148]\times[-0.148,0.148]\times[0,\pi]$. Then we check \eqref{eq:maincond} for $z^*=(0.03705, 0.0488, 0, 0, \pi/2)$. In this case, \eqref{eq:maincond} dose not hold because $f(z^*) \approx 0.11851$ and $f(z^*)-\sum\nolimits_{i=1}^5 d_i C_i \approx -0.21937 < 0.1$. Thus, we must select the maximum length to subdivide $B$ into $B_1$ and $B_2$. In this case, $b_5-a_5=\pi \approx 3.14$ is the maximum. Hence, we divide $B$ into
$B_{1}=[a_{1},b_{1}]\times[a_{1},b_{2}]\times[a_{3},b_{3}]\times[a_{4},b_{4}]\times[a_5,(a_5+b_5)/2]$
and
$B_{2}=[a_{1},b_{1}]\times[a_{1},b_{2}]\times[a_{3},b_{3}]\times[a_{4},b_{4}]\times[(a_5+b_5)/2, b_5]$.
Then we check \eqref{eq:maincond} for $B_1$ and for $B_2$ and repeat this procedure iteratively.

\textbf{}\\
We run the Box-search algorithm \cite{grechuksomam2019convex} by using Matlab$\circledR$ R2018a, see implementation details and Matlab code in Appendix. The programme halts after $n = 527,754,566$ iterations. This rigorously proves that the minimal area $\mathcal{A}(X)$  is greater than $0.1$. Numerically, the program returned value $0.10044$ for this minimal area, with optimal configuration $x_{1}=0.00434,\,y_{1}=0.00648,\,x_{2}=0.00434,\,y_{2}=-0.00434,\,\theta=0.85711$ see Figure \ref{fig:f11}.

\section{Main Theorems}\label{sec:main}

\begin{thm} (Theorem \ref{th:main} in the Introduction)
The area of convex cover $S$ for circle of perimeter $1$, line of length $1/2$, and rectangle of size $0.1727\times0.3273$ is at least $0.1$.
\end{thm}
\begin{proof}
Let $Z$ be a region in Lemma \ref{lem:conditions}. The fact that Box-search algorithm halted together with Lemma \ref{lem:cont} implies that $f(z) > 0.1$ for all $z \in Z$. Then, by Lemma \ref{lem:conditions}, $f(z) > 0.1$ holds for all $z \in {\mathbb R}^5$. Thus, $\mathcal{A}(F,R,L)>0.1$. Since $F\subset C$, $\mathcal{A}(C,R,L) \geq \mathcal{A}(F,R,L)>0.1$.
\end{proof}

\begin{cor}\label{cor:main}
Any convex cover for closed unit curves has area of at least $0.1$.
\end{cor}
\begin{proof}
Let $S$ be a convex cover for closed unit curves. Then $S$ can accommodate $C,R$, and $L$, hence $\mathcal{H}(C,R,L)\subset S$. Thus the area of $S$ is at least $\mathcal{A}(F,R,L)$. However, $\mathcal{A}(F,R,L) > 0.1$ by Theorem \ref{th:main}.
\end{proof}

\begin{figure}[ht]
\begin{centering}
\includegraphics[scale=0.2]{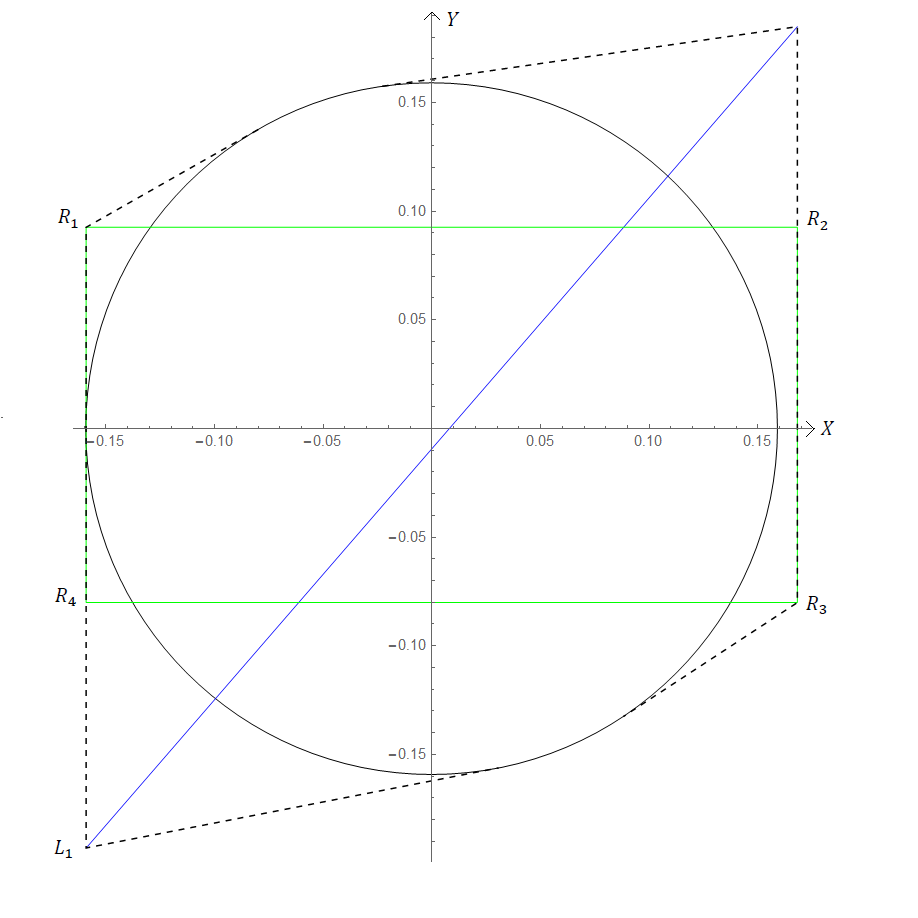}
\par\end{centering}

\caption{The convex hull of the configuration of the minimum
area with $0.10044$ acquired from the box-search algorithm}
\label{fig:f11}
\end{figure}

\section{Conclusion}\label{sec:con}

In this work, we improve the lower bound for the area of convex covers for closed unit arcs from $0.0975$ to $0.1$. First, we used the geometric method to prove Lipschitz bounds for configuration function in $5$ parameters which represent the circle, $0.1727\times0.3273$ rectangle, and line. Next, we used the numerical box-search algorithm developed in \cite{grechuksomam2019convex} to prove the main theorem. In fact, we have used regular $500$-gon in place of circle for convenience of Matlab numerical computations. 

The best bound we can in principle get from circle-line-rectangle configuration is $0.10044$. To improve beyond this, different configurations of objects should be considered. The numerical results in Section \ref{sec:search} suggest that no configuration of three objects can give a bound much better than this. Because considering $4$ and more objects significantly increases the number of parameters and is computationally difficult, it looks like bound $0.1$ (or slightly better) may be the limit of the current technique, and substantially new ideas are required to significantly improve it.

\appendix

\section{Appendix}\label{sec:app}

There are only two functions below which are used in Box-search method. First, function $cvhN2(x_1, x_2, y_1, y_2, \alpha)$ is used to find the area of convex hull for $F, R, L$ which are represented by $x_1, x_2, y_1, y_2, \alpha$. Second, function $checkminNB5$ is used to check the main inequality \eqref{eq:maincond} in a box domain by its parameters $a_1,b_1, \dots, a_5, b_5$. Finally, the Box-search results is shown in A.3.



\subsection{The box's search results}

\begin{figure}
\begin{centering}
\begin{tabular}{|c|c|}
\hline 
Percentage of $r$ & $n$\tabularnewline
\hline 
\hline 
4.5964 \% & 100000000\tabularnewline
\hline 
9.6894 \%  & 200000000\tabularnewline
\hline 
28.4273 \%  & 300000000\tabularnewline
\hline 
95.7067 \%  & 400000000\tabularnewline
\hline 
99.2349  \%  & 500000000\tabularnewline
\hline 
\end{tabular}
\par\end{centering}

\caption{The table of the percentage of $r$ and $n$}\label{fig:table}
\end{figure}

Table \ref{fig:table}  shown the percentage of the area of the initial box for which the inequality \eqref{eq:maincond} is verified and the iteration number every $100,000,000$ steps by Box search program.

\begin{figure}
\begin{centering}
\includegraphics[scale=0.5]{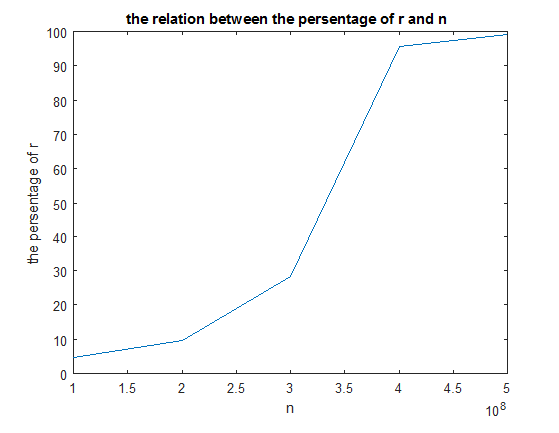}
\par\end{centering}

\caption{The graph of the percentage of $r$ and $n$}\label{fig:rngraph}

\end{figure}

Figure \ref{fig:rngraph} illustrates the graphical process which ups to the number of iterations.

When the program finished, it displayed the result:
[r,n,min,xx1,yy1,xx2,yy2,app] = checkminNB5(0,0.0741,0,0.0976,-0.148,0.148,-0.148,0.148,0,pi,0.1,\\
0.0000001,0,0,0.11,0,0,0,0,0)

r = 0.001990679394226 (this is the area of the initial box, $100\%$ covered)

n = 527754566 (total number of iterations needed)

min = 0.100438196959697 (the minimal area convex hull)

xx1 = 0.004341796875000

yy1 = 0.006481250000000

xx2 = 0.004335937500000

yy2 = -0.004335937500000

app = 0.857111765231102 (the 5 coordinates for the optimal configuration)

\end{document}